 \documentclass[12pt,reqno]{amsart}
\usepackage[usenames,dvipsnames]{xcolor}
\usepackage{amsmath, amsfonts, epsfig, amssymb, graphics}
\usepackage{pgf,tikz}
\usepackage[colorlinks=true,citecolor=cyan]{hyperref}
\usepackage[mathscr]{euscript}

\usepackage{todonotes}

\newcommand{\pref}[1]{{\rm (\ref{#1})}}

\def\area{\mathrm{area}}
\def\diag{\mathrm{diag}}

\def\N{{\mathbb N}}
\def\P{{\mathbb P}}

\def\scalar#1#2{\langle #1,#2\rangle}
\def\Scalar#1#2{\Big\langle #1,#2\Big\rangle}
\def\LargeScalar#1#2{\left\langle #1,#2\right\rangle}

\def\B{B}
\def\SS{S}
\def\Ss{\SS_{m,n}^{(k)}}
\def\Bb{\B_{m,n}^{(k)}}
\def\Sst{\SS_{m,n}^{(k,\ell)}}
\def\Bbt{\B_{m,n}^{(k,\ell)}}

\def\mbf#1{{\mathbf #1}}

\def\auteur#1{{\sc #1}}
\def\titreref#1{{\em #1}}
\def\vol#1{{\bf #1}}
\def\defn#1{\bleu{{\bf #1}}}

\newcommand{\Schroder}{\mathcal{S}}

\def\ou{\"o}

\def\bleu{\textcolor{blue}}
\def\rouge{\textcolor{red}}

\newtheorem{theorem}{\bleu{Theorem}}

\newtheorem{proposition}[theorem]{\bleu{Proposition}}
\newtheorem{corollary}[theorem]{\bleu{Corollary}}

\newtheorem{remark}[theorem]{\bleu{Remark}}

\def\nord#1#2{\draw[red,ultra thick] (#1,#2) -- (#1,#2+1)}
\def\est#1#2{\draw[red,ultra thick] (#1,#2) -- (#1+1,#2)}
\def\nordest#1#2{\draw[red,ultra thick] (#1,#2) -- (#1+1,#2+1)}
\def\triangle#1#2{\filldraw[color=yellow, fill=yellow](#1,#2)-- (1+#1,#2) -- (1+#1,1+#2) -- cycle}
\def\ligneaire#1#2#3{\foreach \x in {1,...,#3} \triangle{#1+\x-1}{#2};}
\def\grille#1#2{\foreach \x in {0,...,#1} \draw[blue] (\x,0) -- (\x,#2);
                        \foreach \x in {0,...,#2}  \draw[blue] (0,\x) -- (#1,\x);
                        \draw[black, thick] (0,0) -- (#1,#2) ;}

\begin{document} 

\title[Schr{\ou}der Combinatorics]{\bleu{Rectangular Schr\"oder Parking Functions Combinatorics}}
\author[J.-C.~Aval]{Jean-Christophe Aval}
\address{LaBRI, CNRS - Universit\'e de Bordeaux,
351 cours de la Lib\'eration, 33405 Talence, France}
\author[F.~Bergeron]{Fran\c{c}ois Bergeron}
\address{D\'epartement de Math\'ematiques, UQAM,  C.P. 8888, Succ. Centre-Ville, 
 Montr\'eal,  H3C 3P8, Canada.}\date{\today. This work was supported by NSERC-Canada}
 \maketitle
\begin{abstract}
We study Schr\"oder paths drawn in a $(m,n)$ rectangle,
for any positive integers $m$ and $n$.
We get explicit enumeration formulas, closely linked to
those for the corresponding $(m,n)$-Dyck paths.
Moreover we study a Schr{\ou}der version of $(m,n)$-parking functions,
and associated $(q,t)$-analogs.

\end{abstract}
 \parskip=0pt

{ \setcounter{tocdepth}{1}\parskip=0pt\footnotesize \tableofcontents}
\parskip=8pt

\section{Introduction}

Schr\"oder numbers (1, 2, 6, 22, 90, 394, 1806, \dots; sequence A006318 in \cite{oeis})
enumerate Schr\ou der paths, defined as paths from $(0, 0)$ to $(n, n)$ made of 
steps (1, 0), (0, 1) and (1, 1), that never go strictly below the diagonal. 
These paths may be seen as a generalization of Dyck paths
(for which only the first two types of steps are allowed), 
which are enumerated by Catalan numbers.
The aim of this work is to investigate properties of such paths embedded
not into a square ($n,n)$, but into a rectangle $(m,n)$.
Partial results are already known, in the case where $m$ and $n$ are coprime,
in particular when $m=rn+1$, which reduces to $mn=n$ \cite{song}.
In the present article, (see Proposition \ref{prop:bizley_schroder})
we obtain a generating series formula, in the general case (no coprimality required). 
As we shall see, the result, as well as the method, is closely related to the case of Dyck paths
studied in~\cite{AB14}.

Moreover, we study Schr\ou der parking functions, defined as 
labeled rectangular Schr\ou der paths, and investigate $q-$ and $(q,t)-$analogs
of our enumeration formulas, in which the parameter $q$ takes into account the area
between the path and the diagonal (a precise definition is given at Subsection \ref{subsec:area}). 


\section{Schr\"oder paths}

\subsection{Schr\"oder polynomials}
Although the general notion of $(m,n)$-Schr\"oder paths considered here seems to be new, the special case of $m=rn+1$ has been considered in~\cite{song}. The ``classical case'' corresponds to $m=n$. In this case, Schr\"oder polynomials, given by the following formula
   \begin{equation}\label{formule}
         \bleu{\SS_n(y)}:=\bleu{\sum _{k=0}^{n} \frac {1}{k+1}\binom{n}{k}\binom{n+k}{n}\,y^k},
            \end{equation}
 with $\SS_0=1$, have a long and interesting history. One of the known facts is that the coefficient of $y^k$, denoted $\SS_n^{(k)}$, enumerates ``Schr\"oder paths'' having $k$ ``diagonal steps'' (relevant definitions are recalled below). The case $k=0$ corresponds to the usual notion of ``Dyck paths'', which are well-known to be counted by Catalan numbers 
    $$\bleu{\SS_n^{(0)} = \frac{1}{n+1}\binom{2n}{n}}.$$ 
Small values of these polynomials are
\begin{eqnarray*}
      \bleu{\SS_1(y)} &=&\bleu{1+y},\\
      \bleu{\SS_2(y)} &=&\bleu{2+ 3\, y+y^2},\\
      \bleu{\SS_3(y)} &=&\bleu{5+10\, y+6\, y^2 +{y}^{3} },\\
      \bleu{\SS_4(y)} &=&\bleu{14+35\,y+30\,{y}^{2}+10\,{y}^{3}+{y}^{4}},\\
       \bleu{\SS_5(y)} &=&\bleu{42+126\,y+140\,{y}^{2}+70\,{y}^{3}+15\,{y}^{4}+{y}^{5}}.
    \end{eqnarray*}
We consider a ``rectangular'' generalization of these polynomials, parametrized by pairs $(m,n)$, with the classical case corresponding to $m=n$. Just as it transpires in the analogous situation for Dyck paths, the case when $m$ and $n$ are coprime is somewhat simpler. To make this more apparent notation-wise, we usually assume that $(m,n)=(ac,bc)$, with $(a,b)$ coprime. Hence $c$ is the greatest common divisor of $m$ and $n$. Thus the $(a,b)$-case corresponds to the coprime situation. We recall that the number of $(a,b)$-Dyck paths, is simply given by the formula
    $$\bleu{\frac{1}{a+b}\binom{a+b}{a}}.$$
The general case is probably best coined in generating series format as 
\begin{equation}\label{bizley}
\bleu{\sum_{d\geq 0}C_{ac,bc}z^d = 
          \exp\!\left(\sum_{j>1}\frac{1}{a}\binom{ja+jb}{ja}\,\frac{z^j}{j}\right)}.
  \end{equation}
The proof of this formula, going back to 1954, is due to Bizley~\cite{bizley}, who attributes it to Grossman~\cite{grossman}. 

\begin{figure}
\begin{center}
\begin{tikzpicture}[scale=.7]
\ligneaire{0}{1}{1};
\ligneaire{0}{2}{2};
\ligneaire{0}{3}{3};
\ligneaire{2}{4}{3};
\ligneaire{2}{5}{4};
\ligneaire{2}{6}{6};
\ligneaire{3}{7}{6};
\ligneaire{7}{8}{3};
\grille{12}{9};
\nord{0}{0};
\nord{0}{1};
\nord{0}{2};
\nordest{0}{3};
\est{1}{4};
\nord{2}{4};
\nord{2}{5};
\nordest{2}{6};
\nordest{3}{7};
\est{4}{8};
\est{5}{8};
\est{6}{8};
\nord{7}{8};
\est{7}{9};
\est{8}{9};
\est{9}{9};
\est{10}{9};
\est{11}{9};
\end{tikzpicture}
\begin{picture}(0,11)(11.5,.8)\setlength{\unitlength}{6.9mm}
\put(8.5,9){$3$}
\put(8.5,8){$6$}
\put(8.5,7){$6$}
\put(8.5,6){$4$}
\put(8.5,5){$3$}
\put(8.5,4){$4$}
\put(8.5,3){$2$}
\put(8.5,2){$1$}
\put(8.5,1){$0$}
\put(7,10){\underline{\bleu{Row areas}}}
\put(-5,9){$0$}
\put(-5,8){$1$}
\put(-5,7){$2$}
\put(-5,6){$3$}
\put(-5,5){$4$}
\put(-5,4){$5$}
\put(-5,3){$6$}
\put(-5,2){$7$}
\put(-5,1){$8$}
\put(-6,10){\underline{\bleu{Row labels}}}
\end{picture}
\end{center}
\vskip-10pt
\caption{The $(12,9)$-Schr\"oder path $000\overline{0}22\overline{2}\overline{3}7$.}\label{fig2}
\end{figure}

\subsection{Schr\"oder path} 
We define \defn{$(m,n)$-Schr\"oder  paths} to be sequences of points $(x,y)$ in $\N\times\N$ such that 
\begin{itemize}\itemsep=4pt
   \item they begin with $(0,0)$ and end with $(m,n)$,
   \item  all such that $my-nx\geq 0$, and
   \item with the next point obtained by either an \defn{up}, \defn{diagonal}, or \defn{right} step.
\end{itemize}
These steps respectively correspond to adding $(0,1)$, $(1,1)$ or $(1,0)$ to a point $(x,y)$. Alternatively, the path could readily be encoded as a word in the letters $u$, $d$ and $r$, with obvious conditions.
We shall denote by $\Ss$ the set of $(m,n)$-Schr\"oder  paths with exactly
$k$ diagonal steps.

Given $m$ and $n$, for each point $(u,v)$ in $\N\times \N$, we define its \defn{$(m,n)$-offset}  to be
\begin{equation}\label{offset}
   d(u,v):=mv-nu.
\end{equation}
Thus, we respectively have $d(u,v)>0$, $d(u,v)=0$ and $d(u,v)<0$, according to the case of $(u,v)$ sitting above, on, or below the diagonal $my=nx$. The \defn{low points} of (any) path are those of minimal off-set
(excluding the origin $(0,0)$). Thus, if a path goes below the diagonal, its low-points are those that sit farthest to the ``south-east'' (this rather depends on $m$ and $n$) of the path.

\subsection{Sequence encoding}
It is often practical to bijectively encode our paths in terms of a sequences $a_0a_1 \cdots a_{n-1}$,  one $a_i$ for each up or diagonal step on the path, reading them from top to bottom. For each up step, we set $a_i=k$ (resp. $a_i=\overline{k}$ for diagonal steps), where $k$ is the number of entire ``cells'' that lie to the left of the unique up (or diagonal) step at height $n-i$. The $a_i$ are said to be the \defn{parts} of $\alpha$. An $a_i$ of the form $\overline{k}$ is said to be \defn{barred}.
In this encoding, $\alpha=a_0a_1 \cdots a_{n-1}$ corresponds to an $(m,n)$-Schr\"oder path if and only if
\begin{itemize}\itemsep=4pt
   \item[(1)] $\bleu{a_0\leq a_1\leq \ldots \leq a_{n-1}}$, (with the order $0<\overline{0}<\ldots k<\overline{k}$),
   \item[(2)]  if $\bleu{a_i=\overline{k}}$, then necessarily $\bleu{a_i<a_{i+1}}$, and
   \item[(3)] for each $i$, we have  $\bleu{a_i\leq \overline{\lfloor i\,m/n\rfloor}}$.
\end{itemize}
Each unbarred $k$, between $0$ and $m$, occurs with some \defn{multiplicity}\footnote{possibly equal to 0.} $n_k$ in a path $\alpha$. Removing $0$-multiplicities, we obtain the  \defn{(multiplicity) composition} $\gamma(\alpha)$ of the sequence $\alpha$, reading these multiplicities in increasing values of $k$. For example, 
   $$ \gamma(0\overline{0}11\overline{1}244\overline{4}) = (1,2,1,2).$$
Clearly $\gamma(\alpha)$ is a composition of $n-k$, where $k$ stands for the number of diagonal steps in $\alpha$. The parts of $\gamma(\alpha)$ may be understood as the lenghts of \defn{risers} in the path. These are maximal sequences of consecutive up-steps.

Any $(m,n)$-Schr\"oder paths may be obtained by either barring or not the rightmost part of a given size in the analogous word encoding of an $(m,n)$-Dyck path. As we will see this makes the enumeration of Schr\"oder easy, once we setup the right tools.

\subsection{Symmetric function weight} 
As we will come to see more clearly later, it is interesting to consider a weighted enumeration of Schr\"oder paths, with the weight lying in the degree graded ring 
   $$\bleu{\Lambda=\bigoplus_{d\geq 0} \Lambda_d}$$
of symmetric ``functions'' (polynomials in a denumerable set of variables $\mbf{x}=x_1,x_2,x_3,\ldots$). Recall that the degree $d$ homogeneous component $\Lambda_d$ affords as a linear basis the set 
    $$\{e_\mu(\mbf{x})\ |\ \mu\vdash d\},$$
  of elementary symmetric functions, with $e_\mu(\mbf{x}):=e_{\mu_1}(\mbf{x})e_{\mu_2}(\mbf{x})\cdots e_{\mu_\ell}(\mbf{x})$ for $\mu=\mu_1\mu_2\ldots \mu_\ell$ running over the set of partitions of $d$. In turn, each factor $e_k(\mbf{x})$ is characterized by the generating function identity
  \begin{equation}
     \bleu{\sum_{k\geq 0} e_k(\mbf{x})z^k=\prod_{i\geq 1} (1+x_i\,z)}.
   \end{equation}
   with $e_0(\mbf{x}):=1$.
It easily follows that
 \begin{equation}
     \bleu{e_k(\mbf{x}+y) = e_k(\mbf{x})+e_{k-1}(\mbf{x})\,y},
   \end{equation}
 where $\mbf{x}+y$ means that we add a new variable $y$ to those occurring in $\mbf{x}$.
 
 With these notions at hand, we now simply set
   \begin{equation}\label{sym_enumer}
      \bleu{\SS_{m,n}(\mbf{x};y):=\sum_{\alpha} \alpha(\mbf{x})\,y^{\diag(\alpha)}},\qquad{\rm with}\qquad \bleu{\alpha(\mbf{x}):=\prod_{k\in\gamma(\alpha)}e_k(\mbf{x})},
      \end{equation}
  and where the sum is over the set of $(m,n)$-Schr\"oder paths $\alpha$, with $\diag(\alpha)$ denoting the number of diagonal steps in $\alpha$. Likewise, we denote by
  \begin{equation}
      \bleu{\Ss(\mbf{x}):=\sum_{\diag(\alpha)=k} \alpha(\mbf{x})},
      \end{equation}
the symmetric function enumerator of $(m,n)$-Schr\"oder paths with exactly $k$ diagonal steps, so that $\SS_{m,n}(\mbf{x};y)=\sum_k \Ss(\mbf{x})\,y^k$.     
  For example, we have,
 \begin{eqnarray*}
 \bleu{\SS_{{1,1}}(\mbf{x};y)}&=&\bleu{e_{{1}}(\mbf{x})+y},\\
\bleu{\SS_{{2,2}}(\mbf{x};y)}&=&\bleu{({e_{{11}}}(\mbf{x})+e_{{2}}(\mbf{x}))+3\,e_{{1}}(\mbf{x})\,y+{y}^{2}},\\
\bleu{\SS_{{3,3}}(\mbf{x};y)}&=&\bleu{({e_{{111}}}(\mbf{x})+3\,e_{{21}}(\mbf{x})+e_{{3}}(\mbf{x}))+ ( 6\,{e_{{11}}}(\mbf{x})+4\,e_{{2}}(\mbf{x}) ) y+6\,e_{{1}}(\mbf{x}){y}^{2}+{y}^{3}}.
\end{eqnarray*}
Observe that, for all $r$ and $n$, we have
\begin{equation}\label{classical}
     \bleu{\SS_{rn+1,n}(\mbf{x};y)=\SS_{rn,n}(\mbf{x};y)},
 \end{equation}
 since the last step of $(rn+1,n)$ must necessarily be a right step, and the coprimality of $rn+1$ and $n$  that staying below the $(rn+1,n)$-diagonal insures as well that we stay below the $(rn,n)$-diagonal. Hence we get the same set of paths.

To make some expressions more compact, we shall use ``plethystic notation'', recalling that we have
\begin{equation}\label{plet_en}
    \bleu{e_{n}[m\,\mbf{x}] :=\sum_{\nu\vdash n} (-m)^{\ell(\nu)} \frac{p_{\nu}(\mbf{x})}{z_\nu} },
\end{equation}
with $a$, $j$ and $m$ considered as ``constants'' for the purpose of plethysm.
For any partition $\nu$ of $n$, of length $\ell(\nu)$, we set
$z_\nu:=1^{d_1}d_1!2^{d_2}d_2!\dots n^{d_n}d_n!$, where $d_i$ is the number of copies of the part $i$ in $\nu$. 

\subsection{Main result}
\begin{proposition}\label{prop:bizley_schroder}
The generating function of rectangular Schr{\ou}der polynomials is given by the following equation:
\begin{equation}\label{eq:bizley-schroder}
\bleu{\sum_{d\geq 0}\SS_{ad,bd}(\mbf{x})\,z^d = 
          \exp\!\left(\sum_{j\geq 1} e_{jb}[ja\,(\mbf{x}+y)]   \,\frac{z^j}{aj}\right)}.
  \end{equation}
\end{proposition}
\begin{proof}[\bf Proof.] 
The proof is inspired from Bizley's original proof \cite{bizley} (see also \cite{AB14}). For fixed $m$, $n$ and $k$, we consider two classes of lattice paths.
The first class is the class $\Ss$ of $(m,n)$-Schr\"oder  paths with exactly $k$ diagonal steps.
The second class, which is denoted by $\B_{m,n}^{(k)}$, consists of similar paths, with same start and end points, finishing with either 
a diagonal or a horizontal step,
but without the above diagonal condition. Naturally, $\B_{m,n}$ stands for the (clearly disjoint) union of the sets  $\B_{m,n}^{(k)}$.
See Figure \ref{fig:b-path}.
\begin{figure}
\begin{center}
\begin{tikzpicture}[scale=.7]
\grille{12}{9};
\est{0}{0};
\est{1}{0};
\nordest{2}{0};
\nord{3}{1};
\nord{3}{2};
\est{3}{3};
\est{4}{3};
\est{5}{3};
\nordest{6}{3};
\nordest{7}{4};
\nord{8}{5};
\est{8}{6};
\nordest{9}{6};
\nord{10}{7};
\est{10}{8};
\nord{11}{8};
\est{11}{9};
\end{tikzpicture}
\end{center}
\caption{An element of $\B_{12,9}^{(3)}$ with 2 low points.}\label{fig:b-path}
\end{figure}

We extend to such paths the symmetric function weight \eqref{sym_enumer} previously only considered on $\Ss$, and naturally set
\begin{equation}
\bleu{\Bb(\mbf{x}):=\sum_{\alpha} \alpha(\mbf{x})},\qquad {\rm and} \qquad
 \bleu{\B_{m,n}(\mbf{x};y):=\sum_k \Bb(\mbf{x})\, y^k},  
\end{equation}
with the first sum over the $\alpha$'s in $\Bb$.
Using the multinomial coefficient notation
   $$\binom{n}{\mu}:=\frac{n!}{(n-d)!\mu_1!\mu_2!\cdots \mu_k!},\qquad {\rm for}\qquad \mu=(\mu_1,\mu_2,\dots,\mu_k)\vdash d\leq n.$$
Let us prove that 
\begin{equation}\label{eq:LB}
\Bb(\mbf{x})=\sum_{\nu\vdash n-k} \binom{m}{k}\binom{m}{d_\nu}\,e_\nu(\mbf{x}).
\end{equation}
To see this, we observe that an element $\alpha$ of $\Bb$ is fully characterized by the following data
\begin{itemize}
\item A partition $\nu$ of $n-k$ which describes the ordered sequence of lengths for vertical risers, giving $\gamma(\alpha)$;
\item The positions of the $k$ diagonal steps among $m$, hence enumerated by the binomial $\binom{m}{k}$;
\item The positions of the parts of $\nu$, counted by the  binomial $\binom{m}{d_\nu}$.
\end{itemize}
Using a Bizley-like argument, exploiting the notion of low points,
we denote by $\Sst$ and $\Bbt$ the subsets of $\Ss$ and $\Bb$ consisting 
of paths with exactly $\ell$ low points ({\it i.e.} those having most negative offset, see~\pref{offset}). Continuing with the logic of our previous notations, we denote by $\Sst(\mbf{x})$  and $\Bbt(\mbf{x})$
the corresponding weighted sums.
We briefly recall the notion of {\em rotation} and refer to \cite{bizley,AB14} for a detailed presentation.
Let us consider an element $\alpha$ of $\Bbt$.
We may cut $\alpha$ at any of its $\ell$ low points, and transpose the two resulting path components.
This operation preserves the number of low points, as well as the risers of $\alpha$.
By this rotation principle, we have a set bijection
\begin{equation}
\Sst\times[m]\simeq\Bbt\times[\ell],
\end{equation}
hence it follows that
$$\sum_\ell {1\over \ell} \Sst(\mbf{x}) = {1\over m} \sum_\ell  \Bbt(\mbf{x}) 
=  {1\over m} \Bb(\mbf{x}).$$
Summing over $k$ gives:
\begin{equation}\label{eq:cle}
\sum_{\ell,k} {1\over \ell} \Sst(\mbf{x}) y^k  = {1\over m} \sum_k \B_{m,n}^{(k)}(\mbf{x})y^k = \frac{1}{m} \B_{m,n}(\mbf{x};y).
\end{equation}
Next, we recall the following expansion of \pref{plet_en}  in terms of elementary symmetric functions
  \begin{equation}\label{eq:enmx}
  \bleu{e_{n}[m\,\mbf{x}]= \sum_{\nu\vdash n} \binom{m}{d_\nu}\,e_\nu(\mbf{x})},
  \end{equation}
  with $d_\nu$ the partition giving the (ordered) multiplicities of the parts of $\nu$. Thus, using \eqref{eq:LB} and \eqref{eq:enmx} we calculate that 
\begin{align}
\B_{m,n}(\mbf{x};y)
&= \sum_k y^k\,\sum_{\nu\vdash n-k} \binom{m}{k}\binom{m}{d_\nu}\,e_\nu(\mbf{x})\nonumber\\
&= \sum_{k=0}^n \binom{m}{k}\,e_{n-k}[m\,\mbf{x}]\, y^k\nonumber\\
&= \sum_{k=0}^n e_{n-k}[m\,\mbf{x}]\,e_k[m\,y]\nonumber \\
&= \bleu{e_m[n(\mbf{x}+y)]}.
\end{align}
We may thus rewrite \eqref{eq:cle} as:
\begin{equation}\label{eq:cle2}
\sum_{\ell,k} {1\over \ell} \Sst(\mbf{x};y)  = {1\over m} e_m[n(\mbf{x}+y)],
\end{equation}
which is the crux of the proof.
We may then set: $S_{m,n}^{(*,\ell)}(\mbf{x},y):=\sum_{k} S_{m,n}^{(k,\ell)}(\mbf{x},y)$
and observe that
\begin{equation}
S_{ac,bc}^{(*,\ell)}(\mbf{x},y)=\sum_{\gamma\models_t c} S_{ac_1,bc_1}^{(*,1)}(\mbf{x},y)\,S_{ac_2,bc_2}^{(*,1)}(\mbf{x},y)\, \cdots\, S_{ac_k,bc_k}^{(*,1)}(\mbf{x},y)
\end{equation}
where the sum is over length $\ell$  compositions $\gamma=({c_1},{c_2},\dots,{c_t})$ of  $c$.
In other terms, if we set 
    $$S_{a,b}^{(*,1)}(\mbf{x},y,z):=\sum_{j=1}^\infty S_{aj,bj}^{(*,1)}(\mbf{x},y) \, z^j,$$ 
then  $S_{ac,bc}^{(*,\ell)}(\mbf{x},y)$ is the coefficient of $z^c$ in ${\big(S_{a,b}^{(*,1)}(\mbf{x},y,z)\big)}^\ell$.
Thus \eqref{eq:cle2} means that ${1\over ac}\,e_{ac}[bc(\mbf{x}+y)]$
is the coefficient of $z^c$ in $-\log(1-S_{a,b}^{(*,1)}(\mbf{x},y,z)))$, whence
\begin{equation}\label{eq:eq3}
S_{a,b}^{(*,1)}(\mbf{x},y,z) =  1-\exp\Big(-\sum_{c=1}^\infty {1\over ac}\,e_{ac}[bc(\mbf{x}+y)]\,z^c\Big).
\end{equation}
We observe that
$$\sum_{d\geq 0}\SS_{ad,bd}(\mbf{x})\,z^d = {1 \over 1-S_{a,b}^{(*,1)}(\mbf{x},y,z) }$$
which, together with \eqref{eq:eq3} gives\eqref{eq:bizley-schroder}.
\end{proof}

For example, for any $a$ and $b$ coprime, we get
\begin{align}
  \bleu{\SS_{a,b}(\mbf{x}+y)}&=\bleu{\frac{1}{a} e_b[a\,(\mbf{x}+y)]},\label{park_ab}\\[6pt]
  \bleu{\SS_{2a,2b}(\mbf{x}+y)}&=\bleu{\frac{1}{2a} e_{2b}[2a\,(\mbf{x}+y)] + \frac{1}{2a^2}\,e_{b}[a\,(\mbf{x}+y)]^2},\\[6pt]
  \bleu{\SS_{3a,3b}(\mbf{x}+y)}&=\bleu{\frac{1}{3a} e_{3b}[3a\,(\mbf{x}+y)] 
                                                      + \frac{1}{2a^2}\,e_{b}[a\,(\mbf{x}+y)] \,e_{2b}[2a\,(\mbf{x}+y)]
                                                      + \frac{1}{6a^3}\,e_{b}[a\,(\mbf{x}+y)]^3}.
\end{align}
Recall from~\cite{AB14} that we have  a Bizley-like formula  
 for the symmetric function enumeration 
  \begin{equation}
     \bleu{C_{m,n}(\mbf{x}):=\sum_{\diag(\alpha)=0} \alpha(\mbf{x})},
  \end{equation}
of $(m,n)$-Dyck paths, namely


\begin{equation}\label{bizley2}
\bleu{\sum_{d\geq 0}C_{ad,bd}(\mbf{x})\,z^d = 
          \exp\!\left(\sum_{j\geq 1} e_{jb}[ja\,\mbf{x}]   \,\frac{z^j}{aj}\right)},
  \end{equation}
This gives us the following formulation of Proposition \ref{prop:bizley_schroder}.
\begin{corollary}\label{prop1}
    For all $m$ and $n$, we have
\begin{equation}\label{cat_schrod} \bleu{\SS_{m,n}(\mbf{x};y) = C_{m,n}(\mbf{x}+y)}.\end{equation}
\end{corollary}

\begin{remark}
We proved \eqref{cat_schrod} by computation of both sides, and comparison.
It would be interesting to get this equality directly by a suited interpretation
of the substitution $\mbf{x}\mapsto\mbf{x}+y$.
\end{remark}

    

As a special case, recalling that we assume $a$ and $b$ to be coprime, we have:
\begin{equation}\label{cas_coprime}
\bleu{\SS_{a,b}(\mbf{x};y)} = \bleu{ \frac{1}{a} e_b[a\,(\mbf{x}+y)]}.\\
\end{equation}
Moreover, we may write the coefficient of $y^k$ in~\pref{cas_coprime} as the following integer coefficient linear combination of the $e_\nu(\mbf{x})$:  
      \begin{equation}\label{formule_coprime}
     \bleu{\SS_{a,b}^{(k)}(\mbf{x})= \sum_{\nu\vdash b-k} \frac{1}{a}\binom{a}{k}\binom{a}{d_\nu}\,e_\nu(\mbf{x})}.
   \end{equation}
Since $\scalar{e_\mu(\mbf{x})}{\sum_{j\geq 0}e_j(\mbf{x})}=1$ for all partition $\mu$, we immediately get
    $$\SS_{a,b}^{(k)}=\Scalar{\SS_{a,b}^{(k)}(\mbf{x})}{\sum_{j\geq 0}e_j(\mbf{x})}.$$
Otherwise stated, for $a$ and $b$ coprime,
  \begin{equation}
     \bleu{\SS_{a,b}^{(k)}= \sum_{\nu\vdash b-k} \frac{1}{a}\binom{a}{k}\binom{a}{d_\nu}}.
   \end{equation}
 In particular, in view of~\pref{classical}, this covers the classical case ($m=n$) as well as the generalized 
 version ($m=rn$) of \cite{song}.
 One also deduces from Proposition~\ref{prop1} the following generalization of a result of Haglund~\cite{haglund}.
\begin{proposition} For all $m$ and $n$,
    \begin{equation}
      \bleu{ \SS_{m,n}^{(k)} = \Scalar{C_{m,n}(\mbf{x})}{e_{n-k}(\mbf{x})h_k(\mbf{x})} }
     \end{equation}
    where $\scalar{-}{-}$ stands for the usual scalar product on symmetric function\footnote{For which $\scalar{p_\mu}{p_\nu}= z_\mu\,\delta_{\mu,\nu}$}.
\end{proposition}
\begin{proof}[\bf Proof.]
We start by recalling the symmetric function identity
\begin{equation}\label{translation}
    \bleu{f(\mbf{x}+y)=\sum_{k\geq 0} y^k\, h_k^\perp\,  f(\mbf{x})},
 \end{equation}
 where $h_k^\perp$ stands for the dual of the operator of multiplication by $h_k(\mbf{x})$ with respect to the symmetric function scalar product. It follows directly from~\pref{cat_schrod} that
 \begin{eqnarray*}
     \sum_k   \SS_{m,n}^{(k)}\,y^k
          &=&\Scalar{C_{m,n}(\mbf{x}+y)}{\sum_{k\geq 0}e_k(\mbf{x})},\\
          &=&\Scalar{\sum_{k\geq 0} y^k\, h_k^\perp\, C_{m,n}(\mbf{x})}{\sum_{j\geq 0}e_j(\mbf{x})},\\
          &=&\Scalar{C_{m,n}(\mbf{x})}{\sum_{k\geq 0} y^k h_k(\mbf{x}) \sum_{j\geq 0}e_j(\mbf{x})},\\
          &=&\sum_{k\geq 0} \Scalar{C_{m,n}(\mbf{x})}{ h_k(\mbf{x}) e_{n-k}(\mbf{x})}\,y^k.
 \end{eqnarray*}
The last equality comes from the fact that $C_{m,n}(\mbf{x})$ is homogeneous of degree $n$, hence all terms of the wrong degree vanish in the scalar product. Evidently we get the announced result by comparing same degree powers of $y$ in both sides of the identity obtained.
\end{proof}    

\subsection{Area enumerator}\label{subsec:area}
 The $i^{\rm th}$ \defn{row area} of a path $\alpha$ in $\Schroder_{n}^{(r)}$, is the integer
    \begin{equation}\label{defn_row_area}
          \bleu{\area_i(\alpha):=\lfloor i\,m/n\rfloor- |a_i|},
      \end{equation}
  where  we set $|\overline{k}|:=k$.
 Summing over all indices $i$, between $1$ and $n$, we get the \defn{area} of $\alpha$:
    \begin{equation} 
              \bleu{\area(\alpha):=\sum_{i=0}^{n-1} \area_i(\alpha)}.
      \end{equation}   
This generalizes to $(m,n)$-Schr\"oder paths a notion of area on Schr\"oder paths introduced for the case $m=n$ in~\cite{bonin} (and further studied by \cite{barcucci}). Following the presentation of \cite{haglund}, this may also be understood as the number of ``upper'' triangles lying above the path and below the diagonal line (as illustrated in Figure~\ref{fig2}). These triangles are also called \defn{area triangles}.
 We have the area $q$-enumerator symmetric function 
   \begin{equation}
      \bleu{ \SS_{m,n}(\mbf{x};y,q)}:=\bleu{\sum_{\alpha} \alpha(\mbf{x})\,q^{\area(\alpha)}y^{\diag(\alpha)}}.
   \end{equation}
   Keeping up with our previous notation conventions, we also set
   \begin{equation}
     \bleu{ \SS_{m,n}^{(k)}(q)}:=\bleu{\sum_{\diag(\alpha)=k}  q^{\area(\alpha)}},\qquad{\rm and}\qquad
      \bleu{C_{m,n}(\mbf{x};q)}:=\bleu{\sum_{\diag(\alpha)=0} \alpha(\mbf{x})\,q^{\area(\alpha)}}.
   \end{equation}
  Observing that the area is independent of whether elements are barred or not, we deduce from Corollary~\ref{prop1} that
  \begin{corollary}\label{cor1} For all $m$ and $n$, we have
   \begin{equation}
      \bleu{\SS_{m,n}(\mbf{x};y,q) = C_{m,n}(\mbf{x}+y;q)},
   \end{equation}
and
       \begin{equation}\label{haglundmn}
        \bleu{\SS_{m,n}^{(k)}(q) = \Scalar{C_{m,n}(\mbf{x};q)}{e_{n-k}(\mbf{x})h_k(\mbf{x})}}.
     \end{equation}
  \end{corollary}
From a result of~\cite{haiman},
it follows that $C_{rn+1,n}(\mbf{x};q)=C_{rn,n}(\mbf{x};q)=\nabla^r(e_n)\big|_{t=1}$, where $\nabla$ is a Macdonald ``eigenoperator'' introduced in~\cite{bergerongarsia}. By this, we mean that its eigenfunctions are the (combinatorial) $q,t$-Macdonald polynomials. Thus, a special instance of~\pref{haglundmn} may be formulated as
 \begin{equation}\label{haglundconj}
     \bleu{\SS_{rn,n}^{(k)}(q)=\Scalar{\nabla^r(e_n)\big|_{t=1}}{ e_{n-k}(\mbf{x})h_k(\mbf{x})}}.
       \end{equation}
 In this way, we get back the case $t=1$ of  Proposition 1 in~\cite{haglund}.
 
\section{Constant term formula} 
The following constant term formula adds an extra parameter to our story. We conjecture that it corresponds to a $(q,t)$-enumeration of $(m,n)$-Schr\"oder parking function, with $t$ accounting for a correctly defined ``dinv''-statistic. 
\begin{equation}\label{negut}
    \bleu{\SS_{m,n}(\mbf{x};y,q,t):= \mathrm{CT}_{z_m,\ldots,z_0}\left(\frac{1}{\mbf{z}_{m,n}}\prod_{i=1}^{m}\frac{z_i(1+y\,z_i)}{z_i-q\,z_{i+1}}\Omega'(\mbf{x};z_i)\prod_{j=i+1}^m\frac{(z_i-z_j)(z_i-qt\,z_j)}{(z_i-qz_j)(z_i-tz_j)}\right)},
\end{equation}
where $\Omega'(\mbf{x};z):=\sum_{k\geq 0} e_k(\mbf{x})\,z^k$, and $\mbf{z}_{m,n}:=\prod_{i=0}^{n-1} z_{\lfloor im/n\rfloor}$. We recall that some care must be used in evaluating multivariate constant term expressions. Indeed,  the order in which successive constant terms are taken does have an impact on the overall result. This is why, in the above formula, the indices appearing after  ``{CT}'' specify that this should be done starting with $z_m$, and then going down to $z_0$.
For example, we have
\begin{eqnarray*}
  \bleu{\SS_{2,2}(\mbf{x};y,q,t)}&=&\rouge{(s_{{2}}+ (q+t)\,s_{{11}})}\bleu{+ ( q+t+1)\, s_{{1}}\,y+{y}^{2}},\\
  \bleu{\SS_{2,3}(\mbf{x};y,q,t)}&=&\rouge{(s_{{21}}+ (q+t)\, s_{{111}})}\bleu{+ ( s_{{2}}+ ( q+t+1) s_{{11}} )\, y+s_{{1}}\,{y}^{2}},\\
  \bleu{\SS_{2,4}(\mbf{x};y,q,t)}&=&\rouge{(s_{{22}}+ (q+t)\, s_{{211}}+ ( {q}^{2}+qt+{t}^{2})\, s_{{1111}})}\\
     &&\qquad\bleu{+ (  ( q+t+1)\, s_{{21}}+ ( {q}^{2}+qt+{t}^{2}+q+t )\, s_{{111}} )\, y}\\
     &&\qquad\bleu{+ ( s_{{2}}+ ( q+t )\, s_{{11}} )\, {y}^{2}}.
\end{eqnarray*}
We underline that formula~\pref{negut} is simply the evaluation at $\mbf{x}+y$ of a similar formula conjectured in~\cite{negut} 
 in relation with $(m,n)$-parking functions. More precisely, it is conjectured in the mentioned paper, that
\begin{equation}
    \bleu{C_{m,n}(\mbf{x};q,t):= \mathrm{CT}_{z_m,\ldots,z_0}\left(\frac{1}{\mbf{z}_{m,n}}\prod_{i=1}^{m}\frac{z_i}{z_i-q\,z_{i+1}}\Omega'(\mbf{x};z_i)\prod_{j=i+1}^m\frac{(z_i-z_j)(z_i-qt\,z_j)}{(z_i-qz_j)(z_i-tz_j)}\right)},
\end{equation}
from which it is clear that \pref{negut} follows, since $\Omega'(\mbf{x}+y;z_i)=(1+y\,z_i)\,\Omega'(\mbf{x};z_i)$.
One may readily show
 that the specialization at $t=1$ of the right-hand side of \pref{negut} does indeed give back our previous $C_{m,n}(\mbf{x}+y;q)=\SS_{m,n}(\mbf{x};y,q)$, since the relevant constant term formula is shown to hold in~\cite{compositionalshuffle}.

This, together with the results and conjectures that appear in~\cite{compositionalshuffle}, opens up a lot of new avenues of exploration. In particular, we may obtain explicit candidates for the $(q,t)$-enumeration of special families of $(m,n)$-Schr\"oder paths (say with return conditions to the diagonal), by the simple device of evaluating at $\mbf{x}+y$  analogous symmetric function formulas for $(m,n)$-Dyck paths. Several questions regarding this are explored in~\cite{positivity}.

\section{Schr\"oder Parking functions} 
An \defn{$(m,n)$-Schr\"oder parking function} is a bijective labeling of the up steps of a $(m,n)$-Schr\"oder path $\alpha$ by the elements of $\{1,2,\ldots,n-\diag(\alpha)\}$. One further imposes the condition that consecutive up steps of same $x$-coordinate have decreasing labels reading them from top to bottom. The path involved in this description is said to be the  \defn{shape} of the parking function. For  $\alpha$ an $(m,n)$-Schr\"oder path, we denote by $\P(\alpha)$ the set of parking function having shape $\alpha$.  
When $\diag(\alpha)=0$, we get the ``usual'' notion of parking functions of shape $\alpha$ (an $(m,n)$-Dyck path).
The $(m,n)$-Schr\"oder parking functions may be understood as preference functions, with some of the parking places being closed to parking (these correspond to diagonal steps). 
For $f\in\P(\alpha)$, we denote by $\area_i(f)$ the row area for the $i^{\rm th}$ row of the shape of $f$. 
Figure~\ref{fig3} gives an example of a parking function of shape $000\overline{0}1\overline{1}223$.

\begin{figure}
\begin{center}
\begin{tikzpicture}[scale=.7]
\foreach \x in {0,...,0}
\filldraw[color=yellow, fill=yellow](0+\x,1)-- (1+\x,1) -- (1+\x,1+1) -- cycle;
\foreach \x in {0,...,1}
\filldraw[color=yellow, fill=yellow](0+\x,2)-- (1+\x,2) -- (1+\x,2+1) -- cycle;
\foreach \x in {0,...,3}
\filldraw[color=yellow, fill=yellow](0+\x,3)-- (1+\x,3) -- (1+\x,3+1) -- cycle;
\foreach \x in {0,...,3}
\filldraw[color=yellow, fill=yellow](1+\x,4)-- (2+\x,4) -- (2+\x,4+1) -- cycle;
\foreach \x in {0,...,3}
\filldraw[color=yellow, fill=yellow](2+\x,5)-- (3+\x,5) -- (3+\x,5+1) -- cycle;
\foreach \x in {0,...,5}
\filldraw[color=yellow, fill=yellow](2+\x,6)-- (3+\x,6) -- (3+\x,6+1) -- cycle;
\foreach \x in {0,...,6}
\filldraw[color=yellow, fill=yellow](2+\x,7)-- (3+\x,7) -- (3+\x,7+1) -- cycle;
\foreach \x in {0,...,6}
\filldraw[color=yellow, fill=yellow](3+\x,8)-- (4+\x,8) -- (4+\x,8+1) -- cycle;
\foreach \x in {0,...,12}
\draw[blue] (\x,0) -- (\x,9);
\foreach \x in {0,...,9}
\draw[blue] (0,\x) -- (12,\x);
\draw[black, thick] (0,0) -- (12,9) ;
\draw[red,ultra thick] (0,0) -- (0,0+1);
\draw[red,ultra thick] (0,1) -- (0,1+1);
\draw[red,ultra thick] (0,2) -- (0,2+1);
\draw[red,ultra thick] (0,3) -- (0+1,3+1);
\draw[red,ultra thick] (1,4) -- (1,4+1);
\draw[red,ultra thick] (1,5) -- (1+1,5);
\draw[red,ultra thick] (2,5) -- (2,5+1);
\draw[red,ultra thick] (2,6) -- (2,6+1);
\draw[red,ultra thick] (2,7) -- (2+1,7+1);
\draw[red,ultra thick] (3,8) -- (3,8+1);
\draw[red,ultra thick] (3,9) -- (3+1,9);
\draw[red,ultra thick] (4,9) -- (4+1,9);
\draw[red,ultra thick] (5,9) -- (5+1,9);
\draw[red,ultra thick] (6,9) -- (6+1,9);
\draw[red,ultra thick] (7,9) -- (7+1,9);
\draw[red,ultra thick] (8,9) -- (8+1,9);
\draw[red,ultra thick] (9,9) -- (9+1,9);
\draw[red,ultra thick] (10,9) -- (10+1,9);
\draw[red,ultra thick] (11,9) -- (11+1,9);
\filldraw[color=blue, fill=white] (.5,.5) circle (.4); \node at (.5,.5) {$1$};
\filldraw[color=blue, fill=white] (.5,1.5) circle (.4); \node at (.5,1.5) {$2$};
\filldraw[color=blue, fill=white] (.5,2.5) circle (.4); \node at (.5,2.5) {$6$};
\filldraw[color=blue, fill=white] (1.5,4.5) circle (.4); \node at (1.5,4.5) {$7$};
\filldraw[color=blue, fill=white] (2.5,5.5) circle (.4); \node at (2.5,5.5) {$3$};
\filldraw[color=blue, fill=white] (2.5,6.5) circle (.4); \node at (2.5,6.5) {$5$};
\filldraw[color=blue, fill=white] (3.5,8.5) circle (.4); \node at (3.5,8.5) {$4$};
\end{tikzpicture}
\end{center}\caption{A Schr\"oder parking function.}
\label{fig3}
\end{figure}
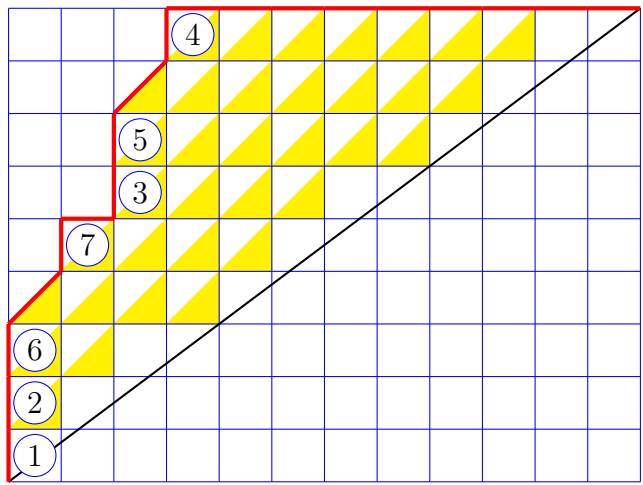

As we did for paths, we consider the $(m,n)$-Schr{\ou}der parking function polynomial
    \begin{equation}
       \bleu{P_{m,n}(y,q)= \sum_k P_{m,n}^{(k)}\,y^k} :=\bleu{\sum _{\alpha} |\P(\alpha)|  
               \,q^{\area(\alpha)}y^{\diag(\alpha)}}.
    \end{equation}
It is easy to derive from Corollary~\ref{cor1}
that
 \begin{corollary} For all $m$ and $n$, we have
   $$\bleu{P_{m,n}(y,q)=\LargeScalar{C_{m,n}(\mbf{x}+y;q)}{\frac{1}{1-p_1(\mbf{x})}}}.$$
Equivalently, for all $k$
          \begin{equation}
        \bleu{P_{m,n}^{(k)}(q)=\Scalar{C_{m,n}(\mbf{x};q)}{p_1(\mbf{x})^{n-k}h_k(\mbf{x})}},
      \end{equation}
 \end{corollary}
 It follows from this, and the observation preceding~\pref{haglundconj}, that
 \begin{equation}
     \bleu{P_{rn,n}^{(k)}(q)=\Scalar{\nabla^r(e_n)\big|_{t=1}}{p_1(\mbf{x})^{n-k}h_k(\mbf{x})}}.
     \end{equation}
Also, for $a$ and $b$ coprime, we have
    \begin{equation}
       \bleu{P_{a,b}^{(k)} =\binom{a}{k} a^{b-k-1}}. 
     \end{equation}

%


\end{document}